\documentclass[twoside, 12pt]{amsart}
\usepackage[letterpaper,hmargin=1.25in,vmargin=1.3in]{geometry}
\usepackage{amscd}
\usepackage{verbatim}
\usepackage{color}
\usepackage{comment}
\usepackage{array}
\usepackage{multirow}
\usepackage{caption}
\usepackage{verbatim}

\input xy
\xyoption{all}

\usepackage{amssymb}
\usepackage{hyperref}
\usepackage{tikz-cd} 
\usepackage{mathtools}

\DeclareMathAlphabet{\cat}{OT1}{cmss}{m}{sl}

\newtheorem*{theorem*}{Theorem}
\newtheorem{theorem}{Theorem}[section]

\newtheorem{proposition}[theorem]{Proposition}

\newtheorem{corollary}[theorem]{Corollary}

\theoremstyle{definition}
\newtheorem*{construction*}{Construction}

\newtheorem{example}[theorem]{Example}

\newcommand{\tens}{\otimes}

\newcommand{\gmu}{\boldsymbol{\mu}}

\newcommand{\ind}{\operatorname{\hspace{0.3mm}ind}}

\newcommand{\disc}{\operatorname{disc}}

\newcommand{\Spec}{\operatorname{Spec}}

\newcommand{\GL}{\operatorname{GL}}

\newcommand{\gSp}{\operatorname{\mathbf{Sp}}}

\newcommand{\gGSpin}{\operatorname{\mathbf{GSpin}}}

\newcommand{\gGSp}{\operatorname{\mathbf{GSp}}}
\newcommand{\gSL}{\operatorname{\mathbf{SL}}}
\newcommand{\gO}{\operatorname{\mathbf{O}}}

\newcommand{\gE}{\operatorname{\mathbf{E}}}
\newcommand{\gGamma}{\operatorname{\mathbf{\Gamma}}}
\newcommand{\gOmega}{\operatorname{\mathbf{\Omega}}}
\newcommand{\gGL}{\operatorname{\mathbf{GL}}}
\newcommand{\gm}{\operatorname{\mathbb{G}}_m}

\newcommand{\gSpin}{\operatorname{\mathbf{Spin}}}

\newcommand{\PGL}{\operatorname{PGL}}

\newcommand{\ed}{\operatorname{ed}}

\newcommand{\rank}{\operatorname{rank}}

\newcommand{\Rep}{\operatorname{Rep}}

\newcommand{\red}{\operatorname{red}}

\renewcommand{\P}{\mathbb{P}}
\newcommand{\Z}{\mathbb{Z}}

\title[Essential dimension of reductive groups]{Essential dimension of reductive groups via generically free representations}

\author
[S.~Baek, Y.~Kim] {Sanghoon Baek, Yeongjong Kim}

\address[Sanghoon Baek]{Department of Mathematical Sciences, 
	KAIST,
	291 Daehak-ro, Yuseong-gu,
	Daejeon 34141,
	Republic of Korea}
\email{sanghoonbaek@kaist.ac.kr}
\urladdr{http://mathsci.kaist.ac.kr/~sbaek/}

\address[Yeongjong Kim]{Department of Mathematical Sciences, 
	KAIST,
	291 Daehak-ro, Yuseong-gu,
	Daejeon 34141,
	Republic of Korea}
\email{kimyj@kaist.ac.kr}

\begin{document}

\maketitle

\begin{abstract}
We provide a simple method to compute upper bounds on the essential dimension of split reductive groups with finite or connected center by means of their generically free representations. Combining our upper bound with previously known lower bound, the exact value of the essential dimension is calculated for some types of reductive groups. As an application, we determine the essential dimension of a semisimple group of classical type or $E_{6}$, and its strict reductive envelope under certain conditions on its center. This extends previous works on simple simply connected groups of type $B$ or $D$ by Brosnan-Reichstein-Vistoli and Chernousov-Merkurjev, strict reductive envelopes of groups of type $A$ by Cernele-Reichstein, and semisimple groups of type $B$ by the authors to any classical type and type $E_{6}$ in a uniform way.
\end{abstract}


\section{Introduction}

Let $G$ be an algebraic group over a field $k$. For a field extension $K/k$, the set of isomorphism classes of $G$-torsors $X\to \Spec(K)$ is bijective to the first non-abelian cohomology set $H^{1}(K,G)$ in the fppf topology. A $G$-torsor $\alpha: Y\to Z$ is called \emph{versal} if every $G$-torsor $X\to \Spec(K)$ with $K$ infinite is isomorphic to the pull-back of $\alpha$ with respect to a point $\Spec(K)\to Z$ and such points forms a dense subset of $Z$. The \emph{essential dimension} $\ed(G)$ of $G$ is defined by the smallest dimension $\dim Z$ of a versal $G$-torsor $Y\to Z$ (see \cite{Reichstein}, \cite{Merkurjev1} for the equivalent definition).

Let $G$ act on a geometrically irreducible variety $X$ over $k$. A subgroup $S$ of $G$ is called a \emph{generic stabilizer} (or \emph{stabilizer in general position}) of $X$ if there is a dense open subset $U$ of $X$ such that the scheme-theoretic stabilizer of every $\bar{k}$-point in $U$ is conjugate to $S$, where $\bar{k}$ denotes the algebraic closure of $k$. We say that a $G$-action $X$ is \emph{generically free} (or \emph{locally free}) if the trivial subgroup of $G$ is a generic stabilizer for this action. Especially, a $G$-representation $V$ is called generically free if $V$ is generically free
as a variety. Indeed, every generically free representation $V$ gives a versal $G$-torsor $U\to Z$ with a non-empty $G$-invariant open subset $U\subset V$, which in turn yields a simple upper bound of $\ed(G)$:
\begin{equation}\label{lem:upperbounds}
\ed(G)\leq \dim V-\dim G.
\end{equation}

For instance, if $G$ is a simple adjoint group, then the adjoint representation of $G$ on the direct sum of two copies of its Lie algebra is generically free, thus $\ed(G)\leq \dim G$. A minimal generically free representation often determines the essential dimension of $G$ including a finite $p$-group \cite{KM} and a split simply connected group $\gSpin(n)$ of type $B$ and $D$ \cite{BRV}, which are remarkable results of this subject.

Let $G$ be a split semisimple group over $k$ and let $G_{\red}$ be a \emph{strict reductive envelope} of $G$, i.e. $G_{\red}$ is a split reductive group such that $G$ is its semisimple part and the center $Z(G_{\red})$ is a torus. Then, $Z(G)=Z(G_{\red})\cap G$ and the embedding of $Z(G)$ into $Z(G_{\red})$ yields an isomorphism between $G_{\red}$ and the cofiber product of $G$ and $Z(G_{\red})$ over $Z(G)$. Indeed, a strict reductive envelope of $G$ arises from the following construction: take an embedding $\phi:Z(G)\to T$ into a split torus $T$ and the quotient group $(G\times T)/Z(G)$,  where $Z(G)$ is embedded into $G\times T$ via $a\mapsto (a, \phi(a)^{-1})$. Then, the quotient become a strict reductive envelope of $G$ whose center is isomorphic to $T$.

For instance, the general linear group $\gGL(n)$ and the even Clifford group $\gGamma^{+}(n)$ are strict reductive envelopes of $\gSpin(n)$ with odd $n$ and $\gSL(n)$, respectively. Though two strict reductive envelopes of $G$ may not be isomorphic, their corresponding classes of torsors are the same (Proposition \ref{prop:twosre}). Therefore, the essential dimension of $G_{\red}$ is independent of the choice of strict reductive envelopes.

In the present paper, we provide a method to compute upper bounds on the essential dimension of an arbitrary semisimple group and its strict reductive envelope by constructing a generically free representation (Theorem \ref{thm:main1}). More precisely, we compute an upper bound of the essential dimension of a semisimple group $G$ and its reductive envelope $G_{\red}$ with $m$ simple components of the form 
\begin{equation}\label{intro:gform}
G=(\tilde{G}_{1}\times \cdots \times \tilde{G}_{m})/\mu
\end{equation}
where each $\tilde{G}_{i}$ is a simple simply connected group and $\mu$ is a central subgroup of $\tilde{G}_{1}\times \cdots \times \tilde{G}_{m}$ with no direct factors of the center of $\tilde{G}_{i}$ for all $1\leq i\leq m$ (i.e., $G$ has no direct factor of an adjoint simple group). For brevity, such a semisimple group is called \emph{reduced}. Indeed, if $G$ has a simple direct factor of a central isogeny $G_{i}'$ of $\tilde{G}_{i}$, i.e., $G=G_{i}'\times G'$ for some semisimple group $G'$, then by \cite[Lemma 1.11]{BerhuyFavi} the sum of the essential dimensions $\ed(G_{i}')$ and $\ed(G')$ gives an upper bound of $\ed(G)$, thus it is essential to compute an upper bound of a reduced semisimple group $G$ in (\ref{intro:gform}). In many cases, our upper bound in Theorem (\ref{thm:main1}) improves the known general upper bound in (\ref{generalbounds}) (see \cite[Proposition 2.1]{CM}, \cite[Example 3.4]{Lotscher}, and \cite[Theorem 4.4, Remark 4.5]{KM}).

On the other hand, there exists a method, as in (\ref{generalbounds}), to compute a lower bound for $\ed(G)$ by considering a central extension of $G$ (see \cite[Theorem 4.1]{Reichstein}, \cite[Theorem 6.2]{Merkurjev1}). In Corollary \ref{cor:cor1} we give a sufficient condition that the upper bound obtained from Theorem \ref{thm:main1} match this lower bound. In Corollary \ref{cor2}, we also provide a complementary method of computing the essential dimension $\ed(G)$, generalizing the proof \cite[Theorem 2.2]{CM} for the case where $G=\gSpin(n)$ with $4\,|\,n$, $n\geq 20$.

Applying our methods to split semisimple groups of types $A$, $B$, $C$, $D$, and $E_{6}$, we obtain new upper bounds and exact values for the essential dimension of the semisimple groups (Propositions \ref{prop:typeBD}, \ref{prop:typeC}, \ref{prop:typeA}, \ref{prop:typeE}) extending previous results on simple groups $\gSpin(n)$ \cite{BRV}, \cite{CR}, strict reductive envelopes of groups of type $A$ \cite{CR}, and semisimple groups of type $B$ \cite{edtypeB}.

In the proofs, we make use of a classification of generically free representations of semisimple groups \cite{Ela1, Ela2, Popov1, Popov2}. Due to lack of generically free representations, the application to the groups of type $A$ is more restrictive than other classical types. However, our approach gives a uniform way to compute the essential dimension of semisimple groups of classical type.

Although we do not explicitly write in the statements, note that in all cases where we compute the exact value of $\ed(G)$ (Corollaries \ref{cor:cor1}, \ref{cor2} and Propositions \ref{prop:typeBD}, \ref{prop:typeC}, \ref{prop:typeA}, \ref{prop:typeE}), $\ed(G)=\ed_p(G)$, where $\ed_{p}(G)$ denotes the essential $p$-dimension of $G$ for a torsion prime of $G$.

In this paper, we write $\gmu_{n}$ and $\gm$ for the group of $n$th roots of unity and the multiplicative group, respectively. For an algebraic group $G$ we denote by $Z(G)$ and $G^{*}$ the center of $G$ and the character group of $G$, respectively. For a positive integer $m$, we denote by $[m]$ the set $\{1,...,m\}$. For a finitely generated group $H$, we denote by $\rank H$ the minimal cardinality of a generating set of $H$. Finally, we denote by $|B|$ the cardinality of a set $B$.

\medskip

\paragraph{\bf Acknowledgements.} 
The authors are grateful to the anonymous reviewers for their valuable comments and suggestions to improve the manuscript. This work was supported by Samsung Science and Technology Foundation under Project Number SSTF-BA1901-02.

\section{Essential dimension of split reductive groups}

In this section, we present a new method (Theorem \ref{thm:main1}) for obtaining upper bounds of the essential dimension of a semisimple group and its strict reductive envelope. In Corollaries \ref{cor:cor1} and \ref{cor2}, we give sufficient conditions for determining the essential dimension based on the upper bounds in the main theorem.

We first show that the essential dimension is independent of the choice of a strict reductive envelope. This is a generalization of \cite[Corollaries A.2, A.3]{CR}. See also \cite[Theorem A.1]{CR} and \cite[Lemmas 6.1, 6.6, 6.11]{Baek} for an explicit description of the classes of torsors of a strict reductive envelope.

\begin{proposition}\label{prop:twosre}
Let $G$ be a split semisimple group over a field $k$ and let $G_{1}$ and $G_{2}$ be two strict reductive envelopes of $G$. Then, $H^{1}(K,G_{1})=H^{1}(K,G_{2})$ for any field extension $K/k$. In particular, $\ed(G_{1})=\ed(G_{2})$.
\end{proposition}
\begin{proof}
Consider the following commutative diagram with exact sequences
\begin{equation}\label{eq:twosre}
\begin{tikzcd}
1 \arrow[r] & Z(G) \arrow[r] \arrow[d, hook, "\varphi_{i}"] & G \arrow[r] \arrow[d, hook] & \bar{G} \arrow[r] \arrow[d, equal] & 1\\
1 \arrow[r] & Z(G_{i}) \arrow[r] & G_{i} \arrow[r] & \bar{G} \arrow[r] & 1,
\end{tikzcd}
\end{equation}		
where $\varphi_{i}$ denotes the embedding for $i=1, 2$ and $\bar{G}$ denotes the adjoint group. Since $Z(G_{i})$ is a torus, by Hilbert Theorem $90$ the diagram (\ref{eq:twosre}) induces the following commutative diagram with exact rows:
\begin{equation*}
\begin{tikzcd}
&& H^1(K,\bar{G})\arrow[r, "\partial"] \arrow[d, equal] & H^2(K, Z(G)) \arrow[d, "(\varphi_{i})_{*}"]\\
1 \arrow[r] & H^1(K, G_{i}) \arrow[r] & H^1(K, \bar{G}) \arrow[r] & H^2(K, Z(G_{i}))
\end{tikzcd}
\end{equation*}
for any field extension $K/k$. Hence, by \cite[Proposition 42]{Serre} the set $H^{1}(K,G_{i})$ is bijective to the kernel of the composition $(\varphi_{i})_{*}\circ \partial$.

Let $(\varphi_{i})^{*}: Z(G_{i})^{*}\to Z(G)^{*}$ be the surjective map dual to $\varphi_{i}$. As the character group $Z(G_{i})^{*}$ is free abelian, there exist $(\psi_{1})^{*}:Z(G_{1})^{*}\to Z(G_{2})^{*}$ and $(\psi_{2})^{*}:Z(G_{2})^{*}\to Z(G_{1})^{*}$ such that $(\varphi_{1})^{*}=(\varphi_{2})^{*}\circ (\psi_{1})^{*}$ and $(\varphi_{2})^{*}=(\varphi_{1})^{*}\circ (\psi_{2})^{*}$, i.e., $\varphi_{1}=\psi_{1}\circ \varphi_{2}$ and $\varphi_{2}=\psi_{2}\circ \varphi_{1}$. Therefore, we have $\ker\big((\varphi_{1})_{*}\circ \partial\big)=\ker\big((\varphi_{2})_{*}\circ \partial\big)$.\end{proof}

For an algebraic group $G$ and a character $\chi:C\to \gm$ of a central subgroup $C$, we denote by $\operatorname{Rep}^{\chi}(G)$ the category of all finite dimensional representation $V$ of $G$ such that $cv=\chi(c)v$ for any $c\in C$ and $v\in V$. We write $n_{G}(\chi)$ for $\gcd_{V\in \operatorname{Rep}^{\chi}(G)} \dim V$. When there is no confusion, we simply write it as $n(\chi)$. For a simple group $G$ and its center $C=Z(G)$, the value of $n_{G}(\chi)$ is given in \cite[Section 4]{Merkurjev2}. For a semisimple group $G$ in (\ref{intro:gform}), a character $\chi:Z(G)\to \gm$ is denoted by $\chi=(\chi_{1},\ldots, \chi_{m})$ together with its restrictions $\chi_{i}$ on $Z(G_{i})$. We write $\operatorname{supp}(\chi)$ for the set of all $i\in [m]$ such that $\chi_{i}\neq 0$.

An alternative way to compute $n(\chi)$ is given in \cite[Theorem 6.1]{Merkurjev1} and \cite[Theorem 4.4, Remark 4.5]{KM}, where $C$ is a diagonalizable group. Namely, consider a central extension $1\to C\to G\to G'\to 1$ over a field $k$ and the induced sequence
\begin{equation*}
H^1(K,G')\xrightarrow[]{\partial}H^2(K,C)\xrightarrow[]{\chi_*}H^2(K,\gm)=\operatorname{Br}(K)
\end{equation*}
for a field extension $K/k$. Then, for a versal $G'$-torsor $E$ we have  
\begin{equation}\label{eq:ngchipartial}
	n(\chi)=\ind\big(\chi_{*}\circ\partial (E)\big),
\end{equation}
where the right-hand side denotes the index of a Brauer class of $\chi_{*}\circ\partial (E)$.

Now we assume that the central subgroup $C$ is isomorphic to $(\gmu_{p})^{r}$ or $(\gm)^{r}$ over a field $k$ of characteristic different from $p$. We further assume that the image $\partial (E)$ is $p$-torsion for some prime integer $p$ in the case where $C\simeq (\gm)^{r}$. We say that a generating set $B$ of the character group $C^{*}$ is \emph{index-minimal} if the sum $\sum_{\chi\in B}n(\chi)$ attains the minimum value over all generating sets of $C^{*}$.  Then, by \cite[Theorem 4.1]{Reichstein}, \cite[Proposition 2.1]{CM}, \cite[Example 3.4]{Lotscher}, and \cite[Proposition 2.1]{CR}, the best known general bounds for the essential dimension of $G$ is given as follows 
\begin{equation}\label{generalbounds}
	\sum_{\chi\in B}n(\chi) -\dim G\,\leq\,  \ed(G)\,\leq\, 	\sum_{\chi\in B}n(\chi)+\ed(G')-\dim C,
\end{equation}
where $B$ is an index-minimal generating set of $C^*$. We remark that the central subgroup $C$ is not necessarily the same for the lower and upper bounds in (\ref{generalbounds}). When $G$ is semisimple but not simple with $C=Z(G)$, one can estimate the essential dimension $\ed(G)$ together with an upper bound of $\ed(\bar{G})$ of the adjoint group $\bar{G}:=G/Z(G)$, although there is a gap of $\ed(\bar{G})+\dim G$ between the upper and lower bound in (\ref{generalbounds}).

In the following theorem, we provide a new upper bound for the essential dimension of a reduced group. This result is a generalization of \cite[Theorem 2.4]{edtypeB} to a semisimple group and its strict reductive envelope.

\begin{theorem}\label{thm:main1}
	Let $G$ be a split semisimple group over a field $k$ with $m$ simple components and let $G_{\red}$ be its strict reductive envelope. Let $B$ be a generating set of $Z(G)^{*}$ of minimal cardinality and let $\rho_{\chi}: G\to \GL(V_{\chi})$ be representations in $\operatorname{Rep}^\chi(G)$ for all $\chi\in B$. Assume that there is a nonempty subset $B_{0}$ of $B$ with $\bigcup_{\chi\in B_{0}}\operatorname{supp}({\chi})=[m]$.
	\begin{enumerate}
		\item[(i)] If the generic stabilizer of $V_\chi$ in $\rho_{\chi}(G)$ is trivial for each $\chi\in B_{0}$, then 
		\begin{equation*}
		\ed(G)\leq \sum_{\chi\in B}\dim V_{\chi}- \dim G.
		\end{equation*}
		
		\item[(ii)]If the generic stabilizer of $\P(V_\chi)$ in $\bar{\rho}_{\chi}(G)$ is trivial for each $\chi\in B_{0}$, where $\bar{\rho}_{\chi}:G\to \PGL(V_{\chi})$ denotes the projective representation induced by $\rho_{\chi}$, then 
		\begin{equation*}
		\ed(G_{\red})\leq \sum_{\chi\in B}\dim V_\chi-\dim G -\rank Z(G).
		\end{equation*}
	\end{enumerate}
\end{theorem}
\begin{proof}
	$(\operatorname{i})$ Let $G=\tilde{G}/\mu$, where $\tilde{G}$ denotes the product of simple simply connected groups $\tilde{G}_{1},\ldots, \tilde{G}_{m}$ and $\mu$ denotes a central subgroup. Let $B$ be a generating set of $Z(G)^{*}$. Consider finite dimensional representations $\rho_{\chi}: G\to \GL(V_{\chi})$ of $\operatorname{Rep}^\chi(G)$ for each $\chi\in B$. We view both $\chi$ and $\rho_{\chi}$ as contained in $Z(\tilde{G})^{*}$ and $\operatorname{Rep}^\chi(\tilde{G})$, respectively.

	We show that the sum $V=\bigoplus_{\chi\in B}V_{\chi}$ is a generically free representation of $G$. Hence, the the inequality in $(\operatorname{i})$ follows from (\ref{lem:upperbounds}). Since the generic freeness of a vector space is not changed under extension of the base field $k$, we may assume that $k$ is algebraically closed.
	
	Let $S_{\chi}$ be a generic stabilizer for the $\tilde{G}$-action on $V_{\chi}$ for all $\chi\in B$. Given two characters $\chi, \lambda\in B$, consider the subgroups
	\[S_{\chi}\cap S_{\lambda}\subset S_{\chi}\subset \tilde{G}.\]
	Since the intersection $S_{\chi}\cap S_{\lambda}$ is a generic stabilizer for the $S_{\chi}$-action on $V_{\lambda}$, it follows from \cite[Proposition 8]{VLPopov} that $S_{\chi}\cap S_{\lambda}$ becomes a generic stabilizer for the $\tilde{G}$-action on $V_{\chi}\bigoplus V_{\lambda}$. Therefore, by an inductive argument on $|B|$, the intersection $S:=\bigcap_{\chi\in B}S_{\chi}$ is a generic stabilizer for the $\tilde{G}$-action on $V$.
	
	Since the center $Z(G)$ is abelian and $k$ is algebraically closed, we have
	\begin{equation}\label{eq:kerchi}
	\bigcap_{\chi\in B}\ker(\chi)=\bigcap_{\chi\in Z(G)^{*}} \ker(\chi)=\mu.
	\end{equation}
	As $\tilde{G}$ is simply connected, the identity component of $\ker(\rho_{\chi})$, denoted by $K_{\chi}$, is given by the product of a subset of the simple factors $\tilde{G}_{i}$, thus we have 
	\begin{equation}\label{eq:kchi}
	K_{\chi}=\ker(\bar{\rho}_{\chi})^{\circ},
	\end{equation}
	where $\ker(\bar{\rho}_{\chi})^{\circ}$ denotes the identity component of $\ker(\bar{\rho}_{\chi})$. It follows from the assumption $\bigcup_{\chi\in B_{0}}\operatorname{supp}({\chi})=[m]$ that for each $i\in [m]$ there exists $\chi\in B_{0}$ such that the image $\rho_{\chi}(G_{i})$ is non-trivial. Therefore, we obtain
	\begin{equation}\label{eq:connectedcomp}
	\bigcap_{\chi\in B_{0}}K_{\chi}=\{1\}, \text{ thus } \bigcap_{\chi\in B_{0}}\ker(\rho_{\chi})\subset Z(\tilde{G}).
	\end{equation}

	Our assumption on the trivial generic stabilizers implies that $S_{\chi}=\ker(\rho_{\chi})$ for all $\chi\in B_{0}$. Hence, by (\ref{eq:connectedcomp}) the generic stabilizer $S$ is central in $Z(\tilde{G})$. As $\rho_{\chi}$ is a representation in $\operatorname{Rep}^{\chi}(\tilde{G})$, we have $S_{\chi}\cap Z(\tilde{G})=\ker(\chi)$, thus by (\ref{eq:kerchi}) we conclude that $S=\mu$, i.e., $V$ is a generically free representation of $\tilde{G}/\mu$.

	\medskip

	\noindent($\operatorname{ii}$) For each $i$, take an embedding of $Z(\tilde{G}_{i})$ to a split torus $T_{i}$ so that we have the induced embedding of $Z(\tilde{G})$ to a split torus $T:=T_{1}\times \cdots \times T_{m}$. Consider strict reductive envelopes 
	\begin{equation*}
	\hat{G}_{i}:=(\tilde{G}_{i}\times T_{i})/Z(\tilde{G}_{i})\,\, \text{ and }\,\, \hat{G}:=(\tilde{G}\times T)/Z(\tilde{G})
	\end{equation*} 
	of $\tilde{G}_{i}$ and $\tilde{G}$, respectively, where $Z(\hat{G}_{i})\simeq T_{i}$ and $Z(\hat{G})\simeq T$. Let $\hat{\chi}$ denote a preimage of $\chi$ under the surjection $Z(\hat{G})^{*}\to Z(\tilde{G})^{*}$ dual to the embedding $Z(\tilde{G})\hookrightarrow  Z(\hat{G})$. Let $T$ act on $V(\chi)$ by scalar multiplication via a character $\hat{\chi}$. Then, togheter with the action of $\tilde{G}$ via $\rho_{\chi}$, we have a representation $\rho_{\hat{\chi}}:\hat{G}\to \GL(V_{\chi})$ whose restriction on $\tilde{G}$ is $\rho_{\chi}$. Set 
	\begin{equation*}
	G_{\red}=\hat{G}/\hat{\mu},\, \text{ where } \hat{\mu}=\bigcap_{\chi\in B}\ker(\hat{\chi}).
	\end{equation*}
	Then, by (\ref{eq:kerchi}) $G_{\red}$ is a strict envelope of $G$.

	Let $S_{\hat{\chi}}$ be a generic stabilizer for the $\hat{G}$-action on $V_{\chi}$. Since for $\chi\in B_{0}$, the generic stabilizer of $\P(V_{\chi})$ is trivial in $\bar{\rho}_{\chi}(G)$, so is in $\rho_{\chi}(G)/Z\big(\rho_{\chi}(G)\big)=\rho_{\hat{\chi}}(\hat{G})/Z\big(\rho_{\hat{\chi}}(\hat{G})\big)$. As $\rho_{\hat{\chi}}\in \Rep^{\hat{\chi}}(\hat{G})$, it follows by \cite[Lemma 3.1]{RV2007} that $S_{\hat{\chi}}=\ker(\rho_{\hat{\chi}})$.

	Now we show that the intersection $\bigcap_{\chi\in B_{0}}\ker(\rho_{\hat{\chi}})$ is central in $\hat{G}$. For each $\chi\in B_{0}$, we set $\hat{G}_{\chi}=\hat{G}/K_{\chi}$ and $\tilde{G}_{\chi}=\tilde{G}/K_{\chi}$. Then, $\hat{G}_{\chi}$ is a strict reductive envelope of $\tilde{G}_{\chi}$ and its center is given by
    \begin{equation}\label{eq:centerhatG}
    Z(\hat{G}_{\chi})\simeq \prod_{G_{i}\not\subset K(\chi)} Z(\hat{G}_{i})\times \prod_{G_{i}\subset K(\chi)} \hat{G}_{i}/G_{i}.
    \end{equation}	
    
    By abuse of notation, we still denote by $\bar{\rho}_{\chi}:\tilde{G}_{\chi}\to \PGL(V_{\chi})$ and $\rho_{\hat{\chi}}: \hat{G}_{\chi}\to \GL(V_{\chi})$ the induced maps. Then, by (\ref{eq:kchi}) the group $\ker(\bar{\rho}_{\chi})$ is central in $\tilde{G}_{\chi}$. Since $Z(\tilde{G}_{\chi})=Z(\hat{G}_{\chi})\cap \tilde{G}_{\chi}$ and $\rho_{\hat{\chi}}\in \Rep^{\hat{\chi}}(\hat{G})$,
    \begin{equation}\label{eq:kerrhocontained}
    \ker(\rho_{\hat{\chi}})\subset Z(\hat{G}_{\chi}).
    \end{equation}
	By assumption, for each $i\in [m]$ there exists $\chi\in B_{0}$ such that $\tilde{G}_{i}\not\subset K_{\chi}$, thus it follows from (\ref{eq:centerhatG}) and (\ref{eq:kerrhocontained}) that $\bigcap_{\chi\in B_{0}}\ker(\rho_{\hat{\chi}})\subset Z(\hat{G})$. Hence,  the same arguments as in (i) show that $V$ is a generically free representation of $G_{\red}$, thus
	 \begin{equation}\label{eq:uppergred}
	\ed(G_{\red})\leq \sum_{\chi\in B} \dim V_\chi - \dim G_{\red}.
	\end{equation}

	It follows from the surjection $Z(G_{\red})^{*}\to Z(G)^{*}$ that $\rank Z(G_{\red})\geq \rank Z(G)$. Since $Z(G_{\red})^{*}=(Z(\hat{G})/\hat{\mu})^{*}$, by the construction of $\hat{\mu}$ the set $\hat{B}:=\{\hat{\chi}\,|\, \chi\in B\}$ generate $Z(G_{\red})^{*}$, thus from the assumption on $B$ i.e., $|B|=\rank Z(G)$, we obtain
	\begin{equation}\label{eq:rankzg}
		\rank Z(G_{\red})=\rank Z(G)=|B|.
	\end{equation}
	As $\dim G_{\red}=\dim G+\rank Z(G_{\red})$, the inequality in (ii) follows from (\ref{eq:uppergred}), (\ref{eq:rankzg}), and Proposition \ref{prop:twosre}.
\end{proof}

The upper bound of Theorem \ref{thm:main1} matches the best known lower bound in (\ref{generalbounds}) under certain conditions:

\begin{corollary}$($cf. \cite[Corollary 6.3]{Merkurjev1}$)$\label{cor:cor1}
In the setting of Theorem \ref{thm:main1}, assume in addition that
for a prime $p\neq \operatorname{char}(k)$,
\begin{enumerate}
	\item[(i)] The $p$-socle $C$ of $Z(G)$ is isomorphic to $(\gmu_{p})^{r}$ for some $r\geq 1$. 
	
	\smallskip
	
	\item[(ii)] $\bar{B}$ is an index-minimal basis of $C^{*}$ such that $\dim V_{\chi}=n(\bar{\chi})$ for all $\chi\in B$,
\end{enumerate}
where $\bar{B}$ and $\bar{\chi}$ denote the images of $B$ and $\chi$ under the restriction homomorphism $Z(G)^{*}\to C^{*}$. Then,	
\begin{equation}\label{cor1:equation}
\ed(G)=\sum_{\chi\in B}\dim V_{\chi}- \dim G.
\end{equation}
Furthermore, if the assumption $(\operatorname{i})$ is replaced by the stronger assumption that $Z(G)$ is a $p$-group, then for any strict reductive envelope $G_{\red}$ of $G$
\begin{equation}\label{cor1:equation2}
\ed(G_{\red})=\ed(G)-\rank Z(G).
\end{equation}
\end{corollary}
\begin{proof}
The equality (\ref{cor1:equation}) immediately follows from (\ref{generalbounds}) and Theorem \ref{thm:main1}. Let $G_{\red}=\hat{G}/\hat{\mu}$ as defined in the proof of Theorem \ref{thm:main1}. For any character $\hat{\chi}$ of $Z(G_{\red})^{*}$, let $\chi$ denote its image under the restriction map $Z(G_{\red})^{*}\to Z(G)^{*}$. Consider the diagram (\ref{eq:twosre}) in which $G_{i}$ and $\varphi_{i}$ are replaced by $G_{\red}$ and the embedding of $G\hookrightarrow G_{\red}$, respectively. Then, this induces the following commutative diagram
\begin{equation}\label{eq:diagramgbar}
\begin{tikzcd}
H^1(K,\bar{G}) \arrow[r, "\partial"] \arrow[d,equal] & H^2(K,Z(G)) \arrow[r, "\chi_{*}"] \arrow[d] & H^2(K,\gm)  \arrow[d, equal] \\
H^1(K,\bar{G}) \arrow[r, "\hat{\partial}"] & H^2(K,Z(G_{\red})) \arrow[r, "\hat{\chi}_{*}"] & H^2(K,\gm),
\end{tikzcd}
\end{equation}
where $K/k$ is a field extension. From the commutativity of the diagram $(\ref{eq:diagramgbar})$ and (\ref{eq:ngchipartial}) it follows that for any $\hat{\chi}\in Z(G_{\red})^{*}$
\begin{equation}\label{eq:nchihatnchibar}
n(\hat{\chi})=n(\chi)\geq n(\bar{\chi}),
\end{equation}
where $\bar{\chi}$ denote the image of $\chi$ under the restriction map $Z(G)^{*}\to C^{*}$.

Let $\hat{B}=\{\hat{\chi}\,|\, \chi\in B\}$ be a generating set of $Z(G_{\red})^{*}$ as defined in the proof of Theorem \ref{thm:main1}. By assumption (ii), the equality in (\ref{eq:nchihatnchibar}) holds for all $\chi\in B$, thus
\begin{equation}\label{eq:minimalsums}
\sum_{\chi\in B}\dim V_{\chi}=\sum_{\chi\in B}n(\hat{\chi})=\sum_{\chi\in B}n(\bar{\chi}).
\end{equation}
As the inequality (\ref{eq:nchihatnchibar}) holds for any $\hat{\chi}\in Z(G_{\red})^{*}$, it follows from the assumption (ii) of the minimality of the sum in (\ref{eq:minimalsums}) that $\hat{B}$ is an index-minimal generating set of $Z(G_{\red})^{*}$.

As $Z(G)$ is a $p$-group, it follows from $(\ref{eq:diagramgbar})$ that $\hat{\partial}(E)$ is $p$-torsion for a versal $\bar{G}$-torsor $E$. Hence, the lower bound in (\ref{generalbounds}) toghether with (\ref{eq:minimalsums}) and (\ref{cor1:equation}) give
\begin{equation*}
	\ed(G_{\red})\geq \sum_{\chi\in B}n(\hat{\chi})-\dim G_{\red}=\ed(G)-\rank Z(G).
\end{equation*}
Therefore, the equality in (\ref{cor1:equation2}) follows by Theorem \ref{thm:main1} and Proposition \ref{prop:twosre}.\end{proof}

In particular, if the group $G$ in Corollary \ref{cor:cor1} is given by a quotient of semisimple group $H$ by a central subgroup isomorphic to $\gmu_{p}$, then the essential dimension of $H$ and its reductive envelope $H_{\red}$ is determined under certain conditions:

\begin{corollary}\label{cor2}
In the setting of Corollary \ref{cor:cor1}, let $1\to \nu\to H\to G\to 1$ be a central extension of $G$ with $\nu\simeq \gmu_{p}$. Let $B\cup\{\omega\}$ be a generating set of $Z(H)^{*}$ whose restriction $\bar{B}\cup\{\bar{\omega}\}$ on the $p$-socle $D$ of $Z(H)$ is an index-minimal basis of $D^{*}$, where the subset $B$ is viewed as contained in $Z(H)^{*}$ by means of the morphism $Z(H)\to Z(G)$. Then,
\begin{equation*}
	\ed(H)=\ed(G)+n_{H}(\check{\omega})\,\, \text{ and }\,\, \ed(H_{\red})=\ed(H)-\rank Z(H),
\end{equation*}
where $\check{\omega}$ denotes the restriction of $\omega$ on $\nu$ and $H_{\red}$ denotes a strict reductive envelope of $H$.
\end{corollary}
\begin{proof}
Let $G_{\red}$ be a strict reductive group as defined in the proof of Theorem \ref{thm:main1}. By assumption, we have $Z(H)\simeq Z(G)\times \nu$ and $D\simeq C\times \nu$. Take an embedding $\nu\hookrightarrow \gm$ so that we get the induced embedding $Z(G)\times \nu\hookrightarrow Z(G_{\red})\times \gm$. Consider a strict reductive envelope $H_{\red}:=\big(H\times Z(G_{\red})\times \gm\big)/Z(H)$ of $H$. Then, we have a central extension
\begin{equation*}
1\to \gm\to H_{\red}\to G_{\red}\to 1
\end{equation*}
and $Z(H_{\red})\simeq Z(G_{\red})\times \gm$.

By abuse of notation, we still denote by $\hat{\chi}$, $\chi$, and $\bar{\chi}$ the characters of $Z(H_{\red})^{*}$, $Z(H)^{*}$, and $D^{*}$ induced by the projection $Z(H_{\red})\to Z(G_{\red})$. Similarly, we denote by $\hat{\omega}$ a preimage of $\omega$ under the surjection $Z(H_{\red})^{*}\to Z(H)^{*}$ dual to $Z(H)\hookrightarrow Z(H_{\red})$.

Let $\check{\omega}$ be the restriction of $\omega$ on $\nu$. As the group $\nu^{*}$ is generated by $\check{\omega}$, it follows from the central extensions of $G$ and $G_{\red}$ and the upper bound in (\ref{generalbounds}) that
\begin{equation*}
\ed(H)\leq n_{H}(\check{\omega})+\ed(G)\,\, \text{ and }\,\, \ed(H_{\red})\leq n_{H}(\check{\omega})+\ed(G_{\red})-1.
\end{equation*}

On the other hand, as $n_{H}(\bar{\chi})\geq n_{G}(\bar{\chi})$ and $n_{H}(\bar{\omega})\geq n_{H}(\check{\omega})$, the lower bounds in (\ref{generalbounds}) and (\ref{cor1:equation}) give the opposite inequality
\begin{equation*}
	\ed(H)\geq  \sum_{\chi\in B}n_{G}(\bar{\chi})+n_{H}(\check{\omega})-\dim H=\ed(G)+n_{H}(\check{\omega})
\end{equation*}
thus the first equation in the statement follows. The same proof of Corollary \ref{cor:cor1} shows that
\begin{equation*}
\ed(H_{\red})\geq \sum_{\chi\in B} n_{H}(\bar{\chi})+n_{H}(\bar{\omega})-\dim H_{\red}\geq \sum_{\chi\in B} n_{G}(\bar{\chi})+n_{H}(\check{\omega})-\dim H_{\red}. 
\end{equation*} 
As $\dim H_{\red}=\dim H + \rank Z(H)=\dim G+\rank Z(H)$, the second equation in the statement immediately follows from (\ref{cor1:equation}) and the first equation.
\end{proof}

For two algebraic groups $G_{1}$ and $G_{2}$, it has been asked whether the equality holds in the following general inequality $\ed(G_1\times G_2)\leq \ed(G_1)+\ed(G_2)$ (see \cite[Lemma 1.11]{BerhuyFavi}). It follows by \cite[Lemma 4.5]{Lotscher} that this equality holds for the groups whose essential dimension can be obtained by Corollary \ref{cor:cor1} or Corollary \ref{cor2}.

\section{Application to groups of classical type and type $E_{6}$}

In this section, we apply Theorem \ref{thm:main1}, and Corollaries \ref{cor:cor1} and \ref{cor2} to calculate the essential dimension of a reduced semisimple group of types $A$, $B$, $C$, $D$, $E_{6}$ and its strict reductive envelope. In particular, we present an application to algebraic theory of quadratic forms in the last subsection. In order to apply our main theorem \ref{thm:main1}, we shall need certain generically free representations of central characters of a given semisimple group. For split semisimple groups over a field of characteristic $0$, the list of non-generically free irreducible (projective) representations are classified in \cite{Ela1, Ela2, Popov1, Popov2}. Part of these classifications will be used for the application.

Let $\rho:G\to \GL(V)$ and $\bar{\rho}:G\to \PGL(V)$ denote an irreducible representation of a split semisimple group $G$ and the projective representation, respectively. If $\P(V)$ is a generically free $\bar{\rho}(G)$-representation, then obviously $V$ is a generically free $\rho(G)$-representation. Conversely, if $V$ is a generically free $\rho(G)$-representation, then by \cite[Theorem 1]{Ela2}, \cite[Theorem]{Popov1} the generic stabilizer of $\P(V)$ in $\bar{\rho}(G)$ is finite and such representations are classified in \cite{Popov1, Popov2}. However, in our application to a semisimple group of a homogeneous Dynkin type as in subsections, there is no distinction between the generic freenesses of $V$ and $\P(V)$, i.e., $V$ is a generically free $\rho(G)$-representation if and only if $\P(V)$ is a generically free $\bar{\rho}(G)$-representation.

We denote by $V(n)$ and $V(n)^{\pm}$ the spin and half spin representations of the spin group $\gSpin(n)$, respectively, where $\dim V(n)=2^{(n-1)/2}$ and $\dim V(n)^{\pm}=2^{(n-2)/2}$. To simplify the notation, we shall use $V(n)^{*}$ to denote $V(n)$ if $n$ is odd and $V(n)^{\pm}$ otherwise. Let $W(n)$ denote the vector representation of $\gSpin(n)$, the symplectic group $\gSp(n)$ with even $n$, and the special linear group $\gSL(n)$.

Let $\gGamma^{+}(n)$ and $\gOmega(n)$ denote the even Clifford group and the extended Clifford group, respectively (see \cite[\S 23, \S 13]{KMRT}). We set
\begin{equation*}
	\gGSpin(n)=\begin{cases}
		\gGamma^{+}(n) & \text{ if } n \text{ is odd},\\
		\gOmega(n) & \text{ if } n \text{ is even}.
	\end{cases}
\end{equation*}
Then, $\gGSpin(n)$ is a strict reductive envelope of $\gSpin(n)$, where the center is given by
\begin{equation*}
	Z(\gGSpin(n))=\begin{cases}
		\gm & \text{ if } n \text{ is odd},\\
		\gm\times \gm & \text{ if } n \text{ is even}.
	\end{cases}
\end{equation*}

The set of isomorphism classes of $\gGSpin(n)$-torsors can be described as follows. For a field extension $K/k$, consider the following set of isomorphism classes of nondegenerate quadratic forms $q$, 
\begin{equation*}
	I^{3}(K,n):=\{q\,\,|\,\, \dim(q)=n,\, \disc(q)=1,\, C(q)=0\},
\end{equation*}
where $\disc(q)$ denotes the discriminant of $q$ and $C(q)$ denotes the Brauer class of the (even) Clifford algebra of $q$ depending on the parity of $n$, i.e., $I^{3}(K,n)$ consists of the classes of $n$-dimensional forms $q$ such that $q\perp \langle -1 \rangle\in I(K)^{3}$ if $n$ is odd and $q\in I(K)^{3}$ otherwise, where $I(K)$ denotes the fundamental ideal of the Witt ring of $K$. For even $n$, we write $PI^{3}(K, n)$ for the set of similarity classes of forms in $I^{3}(K, n)$ (see \cite[\S 6]{CM}). Then, by \cite[\S 3]{CM}, \cite[Lemmas 6.1, 6.11]{Baek} we have a bijection
\begin{equation*}
	H^{1}(K, \gGSpin(n))\simeq\begin{cases}
		I^{3}(K,n) & \text{ if } n \text{ is odd},\\
		PI^{3}(K,n) & \text{ if } n \text{ is even}.
	\end{cases}
\end{equation*}

Let $\gGSp(2n)$ denote the group of symplectic similitudes (see \cite[\S 12]{KMRT}). This is a strict reductive envelope of $\gSp(2n)$ with $Z(\gGSp(2n))=\gm$.

\subsection{Groups of type $B$ and $D$}

Let $\tilde{G}=\prod_{i\in [m]}\gSpin(n_{i})$ be a simply connected group of type $B$ and $D$, where $m\geq 1$ and $n_{i}\geq 3$ for all $i\in [m]$. For each $i\in [m]$, let $V_{i}$ be either a (half)-spin representation $V(n_{i})^{*}$ or the vector representation $W(n_{i})$ with even $n_i$. Consider the tensor product representation
\begin{equation}\label{eq:bdtensorrep}
\rho: \tilde{G}\to \GL(V), \text{ where } V=\bigotimes_{i=1}^{m} V_{i}.
\end{equation}
In the table below, we summarize all groups $\tilde{G}$ having non-generically free $\rho(\tilde{G})$-representations $V$ over all possible representations $\rho$ with any $m\geq 1$, i.e., every such $V$ not in the following table is a generically free $\rho(\tilde{G})$-action.

\begin{center}\label{tab1}
		\captionof{table}{: Non-generically free $\rho(\tilde{G})$-representation \label{table01} }
	
	\footnotesize
	\begin{tabular}{|c | l ||   l |}
		\hline
		\rule{0pt}{2.9ex}  No. &
		$\tilde{G}$ &  V \\ [0.5ex] 
		\hline
		
		\rule{0pt}{2.9ex}    1 &  $\gSpin(n)$,\,\, $3\leq n\leq 16,\,\, n\neq 4,15$ &  $V(n)^{*}$ \\[0.5ex] 
		\hline
		
		\rule{0pt}{2.9ex}   2& $\gSpin(3)\times\gSpin(n)$, $n=3,5,6,7,9,11$  &  $V(3)\tens V(n)^{*}$  \\[0.5ex] 
		
		\hline
		
		\rule{0pt}{2.9ex}   3& $\gSpin(5)\times\gSpin(n)$, $n=5,6,7$ &  $V(5)\tens V(n)^{*}$\\ [0.5ex] 
		\hline
		
		\rule{0pt}{2.9ex}   4 & $\gSpin(6)\times\gSpin(n)$,\,\, $n=6,7,10$ &  $V(6)^{\pm}\tens V(n)^{*}$\\ [0.5ex] 
		\hline
		
		\rule{0pt}{2.9ex}   5 & $\gSpin(3)\times \gSpin(3)\times\gSpin(n)$, \, $n=3,5,6,7$&  $V(3)\tens V(3)\tens V(n)^{*}$\\ [0.5ex] 
		\hline
		
		\rule{0pt}{2.9ex}   6 & $\gSpin(3)\times \gSpin(6)\times\gSpin(n)$, $n=5,6$ &  $V(3)\otimes V(6)^{\pm}\otimes V(n)^{*}$  \\ [0.5ex] 
		\hline
		
		\rule{0pt}{2.9ex}    7 &  $\gSpin(3)^{4}$ &  $V(3)^{\tens 4}$ \\[0.5ex] 
		\hline
		\hline
		
		\rule{0pt}{2.9ex}    8 &  $\gSpin(n)$, \,\,\small$n=2k\geq 6$ & $W(n)$   \\[0.5ex] 
		\hline

		\rule{0pt}{2.9ex}    9 &  $\gSpin(n)\times\gSpin(n)$,\,\, $n=2k\geq 6$ &  $W(n)\tens W(n)$  \\[0.5ex] 
		\hline

		\rule{0pt}{2.9ex}    10 &  $\gSpin(n)\times \tilde{G}$,\,\, $n\neq 4$,\,\, $n=2k>\dim(V)+1$ & $W(n)\tens V$     \\[0.5ex]
		\hline

	\end{tabular}
\end{center}

\medskip

Let us shortly discuss the list on Table \ref{table01}. If $V$ consists of only (half)-spin representations, the list of groups (No. $1\sim7$) immediately follows from \cite[Table 1]{Ela1}, \cite[Theorems 4, 5, 6, 9, Table 6]{Ela2}, \cite[Theorem 1]{Popov1}, and \cite[Theorem 1, Table 0]{Popov2}. Otherwise, $V$ contains at least one vector representation. The group (No.~$8$) is listed in the table as the generic stabilizer of the vector representation of the special orthogonal group $\gO^{+}(n)$ is $\gO^{+}(n-1)$. In the case where $m\geq 2$, the groups $\bar{G}$ (No.~$9\sim10$) is obtained from \cite[Theorems 5, 7, Table 6]{Ela2}, \cite[Theorem 1]{Popov2}.

The character group $Z(\tilde{G})^{*}$ is an abelian group
with components
\begin{equation*}
Z\big(\gSpin(n_{i})\big)^{*}=\begin{cases}
\Z/2\Z & \text{ if } 2\nmid n_{i},\\
\Z/4\Z & \text{ if } 2\,|\, n_{i} \text{ and }4\nmid n_{i},\\
\Z/2\Z\oplus\Z/2\Z & \text{ if } 4\,|\, n_{i}.
\end{cases}
\end{equation*}
A character $\chi$ of $Z(\tilde{G})$ will be denoted by $\chi=(\chi_{1},\ldots, \chi_{m})$, where $\chi_{i}\in Z\big(\gSpin(n_{i})\big)^{*}$. By \cite[Section 4]{Merkurjev2}, for any nontrivial $\chi_{i}$ of $Z\big(\gSpin(n_{i})\big)^{*}$ we have
\begin{equation}\label{nchidimension}
n(\chi_{i})=\begin{cases} 2^{(n_{i}-1)/2} & \text{ if } 2\nmid n_{i} \text{ and } \chi_{i}=1,\\
2^{(n_{i}-2)/2} & \text{ if } 2\,|\, n_{i} \text{ and } \chi_{i}\neq (1,1),\,2,\\
2^{k_{i}} & \text{ if }  2\,|\, n_{i} \text{ and } \chi_{i}=(1,1) \text{ or }2,
\end{cases}
\end{equation}
where $2^{k_{i}}$ is the maximal power of $2$ dividing $n_{i}$.

Let $G=\tilde{G}/\mu$, where $\mu$ is a central subgroup of $\tilde{G}$. For $\chi=(\chi_{1},\ldots, \chi_{m})\in Z(G)^{*}$, regarded as an element of $Z(\tilde{G})^{*}$, we set $V(\chi_{i})$ equal to the trivial representation for any trivial $\chi_{i}$ and set
\begin{equation}\label{setvchi}
 V(\chi_{i})\!=\!\begin{cases}
V(n_{i}) &\text{if } 2\nmid n_{i} \text{ and } \chi_{i}=1,\\
V(n_{i})^{\pm} &\text{if } 2\,|\, n_{i} \text{ and } \chi_{i}\neq (1,1),\, 2,\\
W(n_{i}) &\text{if } 2\,|\, n_{i} \text{ and } \chi_{i}=(1,1)\text{ or }2.
\end{cases}
\end{equation} 
Consider the following tensor product representation of $G$.
\begin{equation}\label{tensor:typeBD}
\rho_{\chi}:G\to \GL(V_{\chi}), \text{ where } V_{\chi}=\bigotimes_{i=1}^{m}V(\chi_{i}).
\end{equation}
Then, we have $V_{\chi}\in \operatorname{Rep}^{\chi}(G)$ as $V(\chi_{i})\in \operatorname{Rep}^{\chi_{i}}(\gSpin(n_{i}))$ for all $1\leq i\leq m$.

Now we apply Theorem \ref{thm:main1} and Corollary \ref{cor:cor1} to semisimple groups of type $B$ and $D$ as follows.

\begin{proposition}\label{prop:typeBD}
Let $G=\big(\prod_{i\in [m]}\gSpin(n_{i})\big)/\mu$ be a reduced group and let $G_{\red}=\big(\prod_{i\in [m]}\gGSpin(n_{i})\big)/\mu$, where $n_{i}\geq 3$ and $\mu$ is a central subgroup. Let $B$ be a subset of $Z(G)^{*}$ whose image $\bar{B}$ is a basis of $Z(G)^{*}/2Z(G)^{*}$ with dimension $|B|$. Assume that there is a subset $B_0$ of $B$ such that $V_{\chi}\neq V$ for all $\chi\in B_0$ and all $V$ in Table $\ref{table01}$ and $\bigcup_{\chi\in B_0}\operatorname{supp}(\chi)=[m]$. Then,
 \begin{equation}\label{prop1:equation}
\ed(G)\leq \sum_{\chi\in B}\dim V_{\chi} - \sum_{i\in [m]}n_{i}(n_{i}-1)/2
\end{equation}
and 
\begin{equation}\label{prop1:equation2}
\ed(G_{\red})\leq \sum_{\chi\in B}\dim V_{\chi} - \sum_{i\in [m]}n_{i}(n_{i}-1)/2-|B|.
\end{equation}

Assume in addition that $\bar{B}$ is index-minimal and each $\chi\in B$ has no component equal to $2$ and has the $i$th component equal to $(1,1)$ only if $n_{i}$ is a power of $2$. Then, the equalities in $($\ref{prop1:equation}$)$ and $($\ref{prop1:equation2}$)$ hold.
\end{proposition}

\begin{proof}
Let $B\subset Z(G)^{*}$ such that the image $\bar{B}$ is a basis of the $2$-socle of $Z(G)^{*}$ with dimension $|B|$. Then, by Nakayama's lemma, $B$ is a generating set of $Z(G)^{*}$ of minimal cardinality. Since $V_{\chi}\neq V$ for all $\chi\in B_0$ and all representations $V$ in Table $\ref{table01}$, each $V_{\chi}$ is a generically free $\rho_{\chi}(G)$-representation, thus the upper bounds in (\ref{prop1:equation}) and (\ref{prop1:equation2}) immediately follow from Theorem \ref{thm:main1}.

If $\chi=(\chi_{1},\ldots, \chi_{m})\in Z(G)^{*}$ has no component equal to $2$ and has the $i$th component equal to $(1,1)$ only if $n_{i}$ is a power of $2$, then by (\ref{nchidimension}), (\ref{setvchi}) we have
\[n(\chi_{i})=n(\bar{\chi}_{i})=\dim V(\chi_{i}),\]
where $\bar{\chi}_{i}$ denotes the restriction of $\chi_{i}$ on the $2$-socle of $\gSpin(n_{i})$. Hence, $n(\bar{\chi})=\prod_{i=1}^{m}n(\chi_{i})=\dim V_{\chi}$. As the image $\bar{B}$ is index-minimal, the equalities in (\ref{prop1:equation}) and (\ref{prop1:equation2}) follow by Corollary \ref{cor:cor1}.\end{proof}

In the special case $m=1$ and $\mu$ is trivial, Proposition \ref{prop:typeBD} immediately shows that
\begin{equation*}
	\ed(\gGSpin(n))=\begin{cases}
		\ed(\gSpin(n))-2 & \text{ if } n\equiv 0 \mod 4,\\
		\ed(\gSpin(n))-1 & \text{ otherwise}
	\end{cases}
\end{equation*}
with the exact value of $\ed(\gSpin(n))$ for any $n\geq 15$. This recovers the results for the computation of $\ed(\gGSpin(n))$ in \cite[Proposition 6.1, Theorem 7.1]{CM} without using the theory of quadratic forms.

Furthermore, Proposition \ref{prop:typeBD} recovers \cite[Theorem 2.4]{edtypeB}. Indeed, if $n_{i}\geq 7$ for all $1\leq i\leq m$ and $G$ has no simple direct factors, then the upper bound (\ref{prop1:equation}) holds without verifying the assumption that $V_{\chi}\neq V$ for all $\chi\in B$ and all $V$ in Table $\ref{table01}$.

\begin{example} Let $G=\big(\prod_{i\in [n]}\gSpin(4a_i+2)\times \prod_{i\in [m]}\gSpin(2b_{i}+1)\big)/\mu$, where $a_i\geq 2$, $b_{i}\geq 1$, $n+m\geq 3$, and $\mu=\{(a_1,\ldots,a_{n+m})\in(\gmu_4)^n\times (\gmu_2)^m\, |\, a_1\cdots a_{n+m}=1\}$. Then, we have
\begin{equation*}
\ed(G)=\prod_{i\in [n]} 4^{a_{i}} \cdot \prod_{i\in [m]} 2^{b_{i}}-\dim G\,\, \text{ and }\,\, \ed(G_{\red})=\ed(G)-1,
\end{equation*}
where $G_{\red}=\big(\prod_{i\in [n]}\gGSpin(4a_i+2)\times \prod_{i\in [m]}\gGSpin(2b_{i}+1)\big)/\mu$.
\end{example}

In the following example, we apply Corollary \ref{cor2} to compute the essential dimension of a semisimple group of type $B$ and $D$, which is not directly covered by Proposition \ref{prop:typeBD}:

\begin{example}
Let $H=\big(\gSpin(4a+2)\times\gSpin(4b+2)\big)/\langle (-1,-1)\rangle$, where $a, b\geq 2$. Consider  $G=\big(\gSpin(4a+2)\times\gSpin(4b+2)\big)/\langle (i,-i)\rangle$, where $i$ denotes a primitive $4$th root of $1$, i.e., $G=H/\mu$ for a central subgroup $\mu:=\langle (i, -i)\rangle/\langle (-1,-1)\rangle$ of $H$. Then, the character groups $Z(G)^{*}=\Z/4\Z$ and $C^{*}=\Z/2\Z$ are generated by $\chi=(1,1)$ and $\bar{\chi}$, respectively, thus by Proposition \ref{prop:typeBD} we have
\begin{equation*}
\ed(G)=4^{a+b}-(2a+1)(4a+1)-(2b+1)(4b+1) \,\, \text{ and }\,\, \ed(G_{\red})=\ed(G)-1,
\end{equation*}
where $G_{\red}=\big(\gGSpin(4a+2)\times\gGSpin(4b+2)\big)/\langle (i,-i)\rangle$.

On the other hand, the character group $Z(H)^{*}\simeq \Z/4\Z\oplus \Z/2\Z$ is generated by $\chi$ and $\lambda=(2,0)$.  Then, the character group $D^{*}$ of the $2$-socle $D$ of $H$ is generated by $\bar{\chi}$ and $\bar{\lambda}$. Since $n(\bar{\chi}+\bar{\lambda})=n(\bar{\chi})$ and $n(\bar{\lambda})=2$, it is an index-minimal basis of $D^{*}$. Therefore, by Corollary \ref{cor2} we obtain
\begin{equation*}
\ed(H)=\ed(G)+2 \,\,\text{ and }\,\, \ed(H_{\red})=\ed(G),
\end{equation*}
where $H_{\red}=\big(\gGSpin(4a+2)\times\gGSpin(4b+2)\big)/\langle (-1,-1)\rangle$.
\end{example}

\subsection{Groups of type $C$}

Let $\tilde{G}=\prod_{i\in [m]}\gSp(2n_{i})$, where $n_{1}\geq \cdots \geq n_{m}\geq 3$ and $m\geq 1$. The character group $Z(\tilde{G})^{*}$ is an elementary abelian $2$-group of rank $m$. Let $\rho$ and $\bar{\rho}$ denote the representation as in (\ref{eq:bdtensorrep}) with $V=\bigotimes_{i=1}^{m} W(2n_{i})$ and the projective representation. Then, by \cite[p.~233]{Popov2} and \cite{Ela2} $V$ (resp. $\P(V)$) is a generically free $\rho(\tilde{G})$ (resp. $\bar{\rho}(\tilde{G})$)-representation if and only if
\begin{equation}\label{genericTypeC}
m\geq 3\, \text{ and }\, 2n_{1}\leq \prod_{i=2}^{m}2n_{i}.
\end{equation}

We calculate below the essential dimension of a semisimple group of type $C$ under the assumption that each $n_{i}$ is a $2$-power and the basis $B$ is index-minimal. We remark that the upper bound in (\ref{eq:edtypeC}) still holds without this assumption.

\begin{proposition}\label{prop:typeC}
Let $G=\big(\prod_{i\in [m]}\gSp(2n_{i})\big)/\mu$ be a reduced group and let $G_{\red}=\big(\prod_{i\in [m]}\gGSp(2n_{i}))\big)/\mu$, where $n_{i}$ is a $2$-power for all $i\in [m]$, $n_{1}\geq \cdots \geq n_{m}\geq 3$, $m\geq 3$, and $\mu$ is a central subgroup. Let $B$ denote an index-minimal basis of $Z(G)^{*}$ and let $\chi_{j}$ denote the first nonzero component of $\chi\in B$. Assume that there is a subset $B_0$ of $B$ such that $\bigcup_{\chi\in B_0}\operatorname{supp}(\chi)=[m]$ and 
\begin{equation*}
	|\operatorname{supp}(\chi)|\geq 3\,\, \text { with }\,\, 2n_{j}\leq \prod_{\chi_{i}\neq 0,\, \chi_{j}}2n_{i}
\end{equation*}
for all $\chi\in B_0$. Then,
\begin{equation}\label{eq:edtypeC}
	\ed(G)= \sum_{\chi\in B}\big(\prod_{\chi_{i}\neq 0} 2n_{i}\big)- \sum_{i=1}^{m}n_{i}(2n_{i}+1) \, \text{ and }\, \ed(G_{\red})=\ed(G)-|B|.
\end{equation}
\end{proposition}
\begin{proof}
Consider the representation $\rho$ as in (\ref{tensor:typeBD}), replacing (\ref{setvchi}) by $V(\chi_{i})=W(2n_{i})$ for all nontrivial $\chi_{i}$. Then, $V_{\chi}\in \operatorname{Rep}^{\chi}(G)$. By our assumption on $\chi$ and (\ref{genericTypeC}), the generic stabilizer of $V_\chi$ (resp. $\P(V_\chi)$) in $\rho_{\chi}(G)$ (resp. $\bar{\rho}_{\chi}(G)$) is trivial for each $\chi\in B_0$, thus by Theorem \ref{thm:main1} the upper bounds in (\ref{eq:edtypeC}) follow. 

Now we assume that $n_{i}$ is a $2$-power for all $i\in [m]$. Then by \cite[Section 4]{Merkurjev2}, $n(\chi_{i})=2n_{i}$ for  $\chi_{i}\neq 0$, thus $n(\chi)=\prod_{\chi_{i}\neq 0} 2n_{i}$. Hence, the equalities in (\ref{eq:edtypeC}) follow by Corollary \ref{cor:cor1}.
\end{proof}

In particular, if $Z(G)^{*}$ is generated by a single element $\chi$ with $\chi_{i}=1$, then 
\begin{equation*}
	\ed(G)= 2^{k_{1}+\cdots +k_{m}}-\sum_{i=1}^{m}2^{k_{i}-1}(2^{k_{i}}+1)\,\, \text{ and }\,\, \ed(G_{\red})=\ed(G)-1,
\end{equation*}
where $k_{1}\leq k_{2}+\cdots +k_{m}$.

\subsection{Groups of type $A$}

Let $\tilde{G}=\prod_{i\in [m]}\gSL(p^{n_{i}})$, where $p$ is a prime integer, $n_{1}\geq \cdots \geq n_{m}\geq 1$, and $m\geq 1$. The character group $Z(\tilde{G})^{*}$ is an abelian $p$-group with components $Z\big(\gSL(p^{n_{i}})\big)^*=\Z/p^{n_{i}}\Z$. Let $\rho$ and $\bar{\rho}$ denote the representation as in (\ref{eq:bdtensorrep}) with $V=\bigotimes_{i=1}^{m} W(p^{n_{i}})$ and the projective representation. Then, by \cite[Theorem 2]{Popov2} and \cite{Ela2}, $V$ (resp. $\P(V))$ is a generically free $\rho(\tilde{G})$ (resp. $\bar{\rho}(\tilde{G})$)-representation if and only if 
\begin{equation}\label{genericTypeA}
	m\geq 3,\,   n_{1}< \sum_{i=2}^{m}n_{i},\, \text{ and }\, \tilde{G}\neq \gSL(2)^{4},\,\, \gSL(3)^{3},\,\, \gSL(p^{n_{1}})^{2}\times \gSL(2).
\end{equation}

In \cite{CR}, the following result was proved for the strict reductive envelope $G_{\red}$ of $G$. Here, we extend this result to the semisimple part $G$.

\begin{proposition}\label{prop:typeA}\cite[Theorem 1.2]{CR}
	Let $G=\big(\prod_{i\in [m]}\gSL(p^{n_{i}})\big)/\mu$ be a reduced group and let $G_{\red}=\big(\prod_{i\in [m]}\gGL(p^{n_{i}})\big)/\mu$, where $n_{1}\geq \cdots \geq n_{m}\geq 1$, $m\geq 3$, and $\mu$ is a central subgroup. Let $B$ denote a subset of $Z(G)^{*}$ such that the image $\bar{B}$ is an index-minimal basis of $Z(G)^{*}/pZ(G)^{*}$ with dimension $|B|$ and each $\chi\in B$ has components $\chi_i$ equal to one of numbers $0,1,p^{n_i}-1$. Let $\chi_{j}$ denote the first nonzero component of $\chi\in B$. Assume that there is a subset $B_0$ of $B$ such that for all $\chi\in B_0$, $|\operatorname{supp}(\chi)|\geq 3$ with
	\begin{equation*}
	n_{j}< \sum_{\chi_i\neq 0,\, i\neq j}n_{i}\,\, \text{ and }\,\, \{ p^{n_{i}}\,|\, \chi_{i}\neq 0\}\neq \{2,2,2,2\}, \{3,3,3\}, \{2^{n_{j}}, 2^{n_{j}}, 2\}	
	\end{equation*}
	as multisets and $\bigcup_{\chi\in B_0}\operatorname{supp}(\chi)=[m]$. Then,
	\begin{equation}\label{eq:typeAed}
		\ed(G)=\sum_{\chi\in B}(\prod_{\chi_{i}\neq 0} p^{n_{i}})- \sum_{i=1}^{m}(p^{2n_{i}}-1) \text{ and } \ed(G_{\red})=\ed(G)-|B|.
	\end{equation}
\end{proposition}
\begin{proof}
	Given $\chi=(\chi_{1},\ldots, \chi_{m})\in B$, we set $V(\chi_{i})=\bigwedge^{\chi_{i}}\big(W(p^{n_{i}})\big)$ for $\chi_{i}=1$ or $p^{n_{i}}-1$. Then, the same argument as in the proof of Theorem \ref{prop:typeC} yields the upper bound in (\ref{eq:typeAed}). As every nontrivial $\chi_{i}$ is equal to either $1$ or $p^{n_{i}}-1$, by \cite[Section 4]{Merkurjev2}
	\begin{equation}
		n(\chi_{i})=p^{n_{i}}, \text{ thus } n(\chi)=\prod_{\chi_{i}\neq 0} p^{n_{i}}.
	\end{equation}
	Therefore, the equalities in (\ref{eq:typeAed}) hold by Corollary \ref{cor:cor1}.
\end{proof}

In particular, if $m\geq 5$ and $Z(G)^{*}$ is generated by a single element $\chi$ with $\chi_{i}=1$ or $p^{n_{i}}-1$, then 
\begin{equation*}
\ed(G)=p^{\sum_{i=1}^{m}n_{i}}- \sum_{i=1}^{m}(p^{2n_{i}}-1)\,\,\text{ and }\,\, \ed(G_{\red})=\ed(G)-1;
\end{equation*}
see \cite[Theorem 1.3(b)]{CR}.

\subsection{Groups of type $E_{6}$} 
Let $G=(\gE_{6})^{m}/\mu$ with $m\geq 2$, where $\gE_{6}$ denote the simply connected group of type $E_{6}$ and $\mu$ is a central subgroup. The character group $Z(\tilde{G})^{*}$ is an elementary abelian $3$-group of rank $m$. 

Let $W^{\pm}$ denote two $27$-dimensional minuscule representations of $\gE_{6}$. Consider the representation $\rho$ as in (\ref{tensor:typeBD}) and the projective representation $\bar{\rho}$, replacing (\ref{setvchi}) by $V(\chi_{i})=W^{\pm}$ for all nontrivial $\chi_{i}$. Then, by \cite[Theorem 3]{Ela2} and \cite[Theorem 1]{Popov2} $V_{\chi}$ (resp. $\P(V_{\chi})$) is generically free $\rho(G)$ (resp. $\bar{\rho}(G)$)-representation if and only if $|\operatorname{supp}(\chi)|\geq 2$. 
Since $n(\chi_{i})=\dim V(\chi_{i})=27$ for a nontrivial $\chi_{i}$ and $V_{\chi}\in \Rep^{\chi}(G)$, Theorem \ref{thm:main1} and Corollary \ref{cor:cor1} yield the following result.

\begin{proposition}\label{prop:typeE}
Let $G=(\gE_{6})^{m}/\mu$ be a reduced group with $m\geq 2$ and let $B$ be an index-minimal basis of $Z(G)^{*}$. If there is a subset $B_0$ of $B$ such that $|\operatorname{supp}(\chi)|\geq 2$ and
$\bigcup\operatorname{supp}(\chi)=[m]$ for all $\chi\in B_0$, then
\begin{equation*}
    \ed(G)=\sum_{\chi\in B}27^{|\operatorname{supp}(\chi)|}-78m\,\, \text{ and }\,\, \ed(G_{\red})=\ed(G)-|B|
\end{equation*}
for a strict reductive envelope $G_{\red}$ of $G$. 
\end{proposition}

In particular, if $Z(G)^{*}$ is generated by a single element $\chi$ with $\chi_{i}=1$, then $\ed(G)=27^m-78m$ and $\ed(G_{\red})=\ed(G)-1$.

\end{document}